\documentclass[review]{article}

\usepackage{lineno,hyperref}
\modulolinenumbers[5]

\usepackage{amsmath, amsthm, amssymb}
\usepackage[ansinew]{inputenc}
\usepackage{amsfonts}
\usepackage{amssymb}
\usepackage{enumitem}
\usepackage{amsthm}
\usepackage{dsfont}
\usepackage{natbib}
\bibpunct[, ]{[}{]}{,}{n}{,}{,}
\makeatother
\theoremstyle{plain}
\DeclareMathAlphabet{\mathbbmsl}{U}{bbm}{m}{sl}

\def\min{\mathop{\rm min}}

\def\d{{\bf d}}

\def\Rr{\mathbb{R}^n}

\def\min{\mathop{\rm min}}

\def\R{\mathbb{R}^{n\times n}}

\def\x{{\bf{x}}}
\def\y{{\bf{y}}}

\newtheorem{theorem}{Theorem}[section]
\newtheorem{lemma}[theorem]{Lemma}
\newtheorem{corollary}[theorem]{Corollary}
\newtheorem{proposition}[theorem]{Proposition}
\newtheorem{example}[theorem]{Example}

\newtheorem{definition}{Definition}
\begin{document}
	\begin{center}
		\large{\bf Generalizations of $R_0$ and SSM properties; Extended Horizontal Linear Complementarity Problem}
	\end{center}\vspace{1.5mm}
	\begin{center}

		\textsc{Punit Kumar Yadav}\\
		Department of Mathematics\\ Malaviya National Instiute of Technology, Jaipur, 302017, India\\
		E-mail address: punitjrf@gmail.com\\
		\textsc{K. Palpandi}\\
		Department of Mathematics\\ Malaviya National Instiute of Technology, Jaipur, 302017, India\\
		E-mail address: kpalpandi.maths@mnit.ac.in
	\end{center}
\begin{abstract}
	In this paper, we first introduce  $R_0$-$W$ and ${\bf SSM}$-$W$ property for the set of matrices which is a generalization of $R_0$ and the strictly semimonotone matrix. We then prove some existence results for the extended horizontal linear complementarity problem when the involved matrices have these properties. With an additional condition on the set of matrices,  we prove that the ${\bf SSM}$-$W$ property is equivalent to the unique solution for the corresponding extended horizontal linear complementarity problems. Finally, we give a necessary and sufficient condition for the connectedness of the solution set of the extended horizontal linear complementarity problems.
\end{abstract}

\section{Introduction} 
The standard linear complementarity problem (for short LCP),  LCP($C,q$), is to find  vectors $x,y$ such that \begin{equation}
x\in \Rr, ~y=Cx + q\in\Rr ~\text{and}~ x\wedge y = 0,\end{equation} where $C\in \R,~q\in \Rr$ and $'\wedge'$ is a min map. The LCP has numerous applications in numerous domains, such as optimization, economics, and game theory. Cottle and Pang's monograph \cite{LCP} is the primary reference for standard LCP. Various generalisations of the linear complementarity problem have been developed and discussed in the literature during the past three decades (see, \cite{n2n,elcp,telcp,hlcp,hlcpm,rhlcp}). The extended horizontal linear complementarity problem is one of the most important extensions of LCP, which various authors have studied; see \cite{PP0,exis,n2n} and references therein. For a given ordered set of matrices ${\bf C}:=\{C_0,C_1,...,C_k\} \subseteq \mathbb{R}^{n \times n}$, vector $q\in \mathbb{R}^n$ and ordered set of positive vectors ${\bf d}:=\{d_{1},d_{2},...,d_{k}\} \subseteq \mathbb{R}^n $, the extended horizontal linear complementarity problem (for short EHLCP), denoted by EHCLP(${\bf C},{\bf d},q$), is to find a vector $x_{0},x_{1},...,x_{k} \in \mathbb{R}^n$ such that \begin{equation}\label{e1}
\begin{aligned}
	C_0 x_{0}=&q+\sum_{i=1}^{k} C_ix_{i},\\
	x_{0}\wedge x_{1}=0 ~~\text{and} ~~ (d_{j}-&x_{j})\wedge x_{j+1}=0, ~1\leq j\leq k-1.\\
\end{aligned}
\end{equation} 
If $k=1$, then EHLCP becomes the horizontal linear complementarity problem (for short HLCP), that is,  \begin{equation*}
\begin{aligned}
	C_0 x_{0}-C_1x_{1}=q~~\text{and}~~x_{0}\wedge x_{1}=0.
\end{aligned}
\end{equation*} 
Further, HLCP reduces to the standard LCP by taking $C_0 =I$. Due to its widespread applications in numerous domains, the horizontal linear complementarity problem has received substantial research attention from many academics; see \cite{hlcp,hlcpm,rhlcp, homo} and reference therein.

Various writers have presented new classes of matrices for analysing the structure of LCP solution sets in recent years; see for example, \cite{LCP, fvi,PP0}.
The classes of  $R_0$, $P_0$,  $P$, and strictly semimonotone (SSM) matrices play a crucial role in the existence and uniqueness of the solution to LCP. For instance, $P$ matrix (if [$x\in\Rr,x*Ax\leq 0\implies x=0$]) gives a necessary and sufficient condition for the uniqueness of the solution for the LCP (see, Theorem 3.3.7 in \cite{LCP}). To get a similar type of existence and uniqueness results for the generalized LCPs, the notion of $P$  matrix was extended for the set of matrices as the column $W$-property by Gowda et al. \cite{PP0}. They proved that column $W$-property gives the solvability and the uniqueness for the extended horizontal linear complementarity problem (EHLCP). Also, they have generalized the concept of the  $P_0$-matrix as the column $W_0$-property.  

Another class of matrix, the so-called SSM matrix,  has importance in LCP theory. This class of matrices provides a unique solution to LCP on $\Rr_+$ and also gives the existence of the solution for the LCP (see, \cite{LCP}). For a $Z$ matrix (if all the off-diagonal entries of a matrix are non-positive), $P$ matrix is equivalent to the SSM matrix (see, Theorem 3.11.10 in \cite{LCP}). A natural question arises whether the SSM matrix can be generalized for the set of matrices in the view of EHLCP and whether we have a similar equivalence relation for the set of $Z$ matrices. In this paper, we would like to answer this question.

The connectedness of the solution set of LCP has a prominent role in the study of the LCP. We say a matrix is connected if the solution set of the corresponding LCP is connected. 
In \cite{ctd}, Jones and Gowda addressed the connectedness of the solution set of the LCP. They proved that the matrix is connected whenever the given matrix is a $P_0$ matrix and the solution set has a bounded connected component. Also, they have shown that if the solution set of LCP is connected, then there is almost one solution of LCP for all $q>0.$ Due to the specially structured matrices involved in the study of the connectedness of the solution to LCP, various authors studied the connectedness of LCP, see for example \cite{ctd,ctdl,Cntd}. The main objectives of this paper are to answer the following questions:
\begin{itemize}
\item[(Q1)]  In LCP theory, it is a well-known result that the $R 0$ matrix gives boundedness to the LCP solution set. The same holds true for HLCP \cite{szn}. This motivates the question of whether or not the notion of $R_0$ matrix can be generalized to the set of matrices. If so, then can we expect the same kind of outcome in the EHLCP?
\item [(Q2)] Given that a strictly semimonotone matrix guarantees the existence of the LCP solution and its uniqueness for $q\geq 0$, it is natural to wonder whether the concept of SSM matrix can be extended to the set of matrices. If so, then whether the same result holds true for EHLCP.
\item [(Q3)] Motivated by the results of Gowda and Jones \cite{ctd} regarding the connectedness of the solution set of LCP, one can ask whether the solution set of EHLCP is connected if the set of matrices has the column $W_0$ property and the solution set of the corresponding EHLCP has a bounded connected component.
\end{itemize}
The paper's outline is as follows: We present some basic definitions and results in section 2. We generalize the concept of $R_0$ matrix and prove the existence result for EHLCP in section 3. In section 4, we introduce the {\bf SSM}-$W$ property, and we then study an existence and uniqueness result for the EHLCP when the underlying set of matrices have this property. In the last section, we give a necessary and sufficient condition for the connectedness of the solution set of the EHLCP.

\section{Notations and Preliminaries}
\subsection{Notations}Throughout this paper, we use the following notations: \begin{itemize}
\item[(i)] The  $n$ dimensional Euclidean space with the usual inner product will be denoted by $\mathbb{R}^n$. The set of all non-negative vectors (respectively, positive vectors) in $\mathbb{R}^n$ will be denoted by $\mathbb{R}^n_+$ (respectively, $\mathbb{R}^n_{++}$ ). We say $x \geq 0$ (respectively, $ >0$) if and only if $x\in\mathbb{R}^n_+$ (respectively, $\mathbb{R}^n_{++})$.
\item [(ii)] The $k$-ary Cartesian power of $\Rr$ will be denoted by $\Lambda^{(k)}_n$ and the $k$-ary Cartesian power of $\Rr_{++}$ will be denoted by $\Lambda^{(k)}_{n,++}$. The bold zero '${\bf 0}$' will be used for denoting the  zero vector $(0,0,...,0)\in \Lambda^{(k)}_n.$ 
\item [(iii)] The set of all $n\times n$ real matrices will be denoted by $\R$. We use the symbol $\Lambda^{(k)}_{n\times n}$ to denote the $k$-ary Cartesian product of  $\R$. 
\item [(iv)] We use $[n]$ to denote the set $\{1,2,...,n\}$.
\item [(v)] Let $M\in\R$. We use $\text{diag}(M)$ to denote  the vector $(M_{11},M_{22},...,M_{kk})\in \Rr$, where $M_{ii}$ is the $ii^{\rm th}$ diagonal entry of matrix $M$ and  $\text{det}(M)$ is used to denote the determinant of matrix $M$.
\item[(vi)] SOL(${\bf C}, \d, q$) will be used for denoting the set of all solution to EHLCP(${\bf C},\d,q$).
\end{itemize} 
We now recall some definitions and results from the LCP theory, which will be used frequently in our paper.

\begin{proposition}[\cite{wcp}]\label{star}
Let $V=\mathbb{R}^n.$ Then, the following statements are equivalent.
\begin{itemize}
	\item [\rm(i)] $x\wedge y=0.$
	\item [\rm(ii)] $x,y\geq 0$ and $~x*y=0,$ where $*$ is the Hadamard product.
	\item[\rm(iii)] $x,y\geq 0~\text{and}~\langle x,y\rangle=0.$
\end{itemize}	\end{proposition} 
\begin{definition}[\cite{PP0}]\rm
Let ${\bf C}=(C_0,C_1,...,C_k)\in\Lambda^{(k+1)}_{n\times n}$. Then a matrix $R\in\R$ is column representative of ${\bf C}$ if $$R._j\in\big\{(C_0)._j,(C_1)._j,...,(C_k)._j\big\},~\forall j\in[n],$$
where $R._j$ is the $j^{{\rm th}}$ column of matrix $R.$
\end{definition}
Next, we define the column W-property.
\begin{definition}[\cite{PP0}] \rm
Let  ${\bf C}:=(C_0,C_1,...,C_k)\in\Lambda^{(k+1)}_{n\times n}$. Then we say that ${\bf C}$ has the 
\begin{itemize}
	\item[\rm (i)] {\it column $W$-property} if the determinants of all the column representative matrices of ${\bf C}$ are all positive or all negative.
	\item[\rm (ii)] {\it column $W_0$-property} if there exists ${\bf N}:=(N_0,N_1,...,N_k)\in \Lambda^{(k+1)}_{n\times n}$ such that ${\bf C+\epsilon N}:=(C_0+ \epsilon N_0,C_1+\epsilon N_1,...,C_k+\epsilon N_k)$ has the column $W$-property for all $\epsilon>0$.
\end{itemize}  
\end{definition}
Due to Gowda and Sznajder \cite{PP0}, we have the following result.
\begin{theorem}[\cite{PP0}] \label{P1}
For ${\bf C}=(C_0,C_1,...,C_k)\in\Lambda^{(k+1)}_{n\times n}$, the following are equivalent:\begin{itemize}
	
	\item[\rm(i)]${\bf C}$ has the column $W$-property.
	\item[\rm(ii)] For arbitrary non-negative diagonal matrices $D_{0},D_{1},...,D_{k}\in\R$
	with $\text{\rm diag}(D_{0}+D_{1}+D_{2}+...+D_{k})>0$,
	$$\text{\rm det}\big(C_0D_{0}+C_1D_{1}+...+C_kD_{k}\big)\neq 0.$$
	\item [\rm(iii)]$C_0$ is invertible and  $(I,C_0^{-1}C_1,...,C_0^{-1}C_k)$ has the column $W$-property.
	\item [\rm(iv)]  For all $q\in\Rr$ and $\d\in\Lambda^{(k-1)}_{n,++}$, {\rm EHLCP}$({\bf C},\d,q)$ has a unique solution. \end{itemize}
\end{theorem}
If $k=1$ and $C_0^{-1}$ exists, then   HLCP($C_0,C_1,q$) is equivalent to LCP($C_0^{-1}C_1,C_0^{-1}(q)$). In this case, $C_0^{-1}C_1$ is a $P$  matrix if and only if for all $q\in\Rr$,  LCP($C_0^{-1}C_1,C_0^{-1}(q)$) has a unique solution (see, Theorem 3.3.7 in \cite{LCP}). Hence we have the following theorem given the previous theorem.
\begin{theorem}[\cite{PP0}]\label{C1}
Let $(C_0,C_1)\in\Lambda^{(2)}_{n\times n}$. Then the following are equivalent.
\begin{itemize}
	\item [\rm(i)]  $(C_0,C_1)$ has the column $W$-property.
	\item [\rm(ii)] $C_0$ is invertible and $C_0^{-1}C_1$ is a $P$  matrix.
	\item [\rm(iii)]  For all $q\in\Rr$, {\rm HLCP}$(C_0,C_1,q)$ has a unique solution.
\end{itemize}
\end{theorem}
\subsection{Degree theory} 
We now recall the definition and some properties of a degree from \cite{fvi,deg} for our discussion. 

Let $\Omega$ be an open bounded set in $\Rr$. Suppose $h:\bar{\Omega}\rightarrow \Rr$ is a continuous  function and a vector $p\notin h(\partial\Omega)$, where $\partial\Omega$ and $\bar{\Omega}$ denote the  boundary and closure of $\Omega$, respectively. Then the degree of $h$ is defined with respect to $p$ over  $\Omega$  denoted by $\text{deg}(h,\Omega,p).$ The equation $h(x)=p$ has a solution whenever $\text{deg}(h,\Omega,p)$ is non-zero. If $h(x)=p$ has only one solution, say $y$ in $\Rr$, then the degree is the same overall bounded open sets containing $y$. This common degree is denoted by $\text{deg}(h,p)$.
\subsubsection{Properties of the degree} The following properties are used frequently here.
\begin{itemize}
\item[(D1)] deg($I,\Omega,\cdot)=1$, where $I$ is the identity function.
\item [(D2)]    {\bf Homotopy invariance}: Let a homotopy $\Phi(x,s):\Rr\times[0,1]\rightarrow \Rr $  be continuous. 
If the zero set of $\Phi(x,s),~X=\{x:\Phi(x,s)={0}~\text{for some}~s\in[0,1]\}$ is bounded, then for any bounded open set $\Omega$ in $\Rr$ containing the zero set $X$, we have $$\text{deg}(\Phi(x,1),\Omega,{ 0})=\text{deg}(\Phi(x,0),\Omega,{0}).$$
\item[(D3)] {\bf Nearness property}: Assume $\text{deg}(h_1(x),\Omega,p)$ is defined and $h_2:{\bar\Omega}\rightarrow \Rr$ is a continuous function. If 
$\text{sup}\displaystyle_{x\in\Omega}\| h_2(x)-h_1(x)\|<\text{dist}(p,\partial\Omega)$, then $\text{deg}(h_2(x),\Omega,p)$ is defined and equals to $\text{deg}(h_1(x),\Omega,p)$.
\end{itemize}
The following result from Facchinei and Pang  \cite{fvi} will be used later.
\begin{proposition}[\cite{fvi}]\label{ND}
Let $\Omega$ be a non-empty, bounded open subset of $\Rr$
and let $\Phi:\bar{\Omega}\rightarrow \Rr$ be a continuous injective mapping.  Then $\text{\rm deg}(\Phi,\Omega,p)\neq0$ for all
$p\in\Phi(\Omega)$.
\end{proposition}
\noindent{\bf Note}: All the degree theoretic results and concepts are also applicable over any finite dimensional Hilbert space (like $\Rr$ or $\Rr\times\Rr\times\Rr$ etc).
\section{$R_0$-$W$ property} In this section, we first define the $R_0$-$W$ property for the set of matrices which is a natural generalization of $R_0$ matrix in the LCP theory. We then show that the $R_0$-$W$ property gives the boundedness of the solution set of the corresponding EHLCP.
\begin{definition}\rm
Let ${\bf C}=(C_0,C_1,...,C_k) \in \Lambda^{(k+1)}_{n\times n}$.  We say that ${\bf C}$ has the $R_0$-$W$ {\it property} if the system 
$$C_0x_{0}=\sum_{i=1}^{k} C_ix_{i}~\text{and}~x_{0}\wedge x_{j}=0 ~~\forall~j\in [k]$$ has only zero solution.

\end{definition}
It can be seen easily that the $R_0$-$W$ property coincides with $R_0$ matrix when $k=1$ and $C_0 =I$.  Also it is noted (see, \cite{wcp}) that if $k=1$, then the $R_0$-$W$ property referred as $R_0$ pair. To proceed further, we prove the following result.
\begin{lemma}\label{l1}   Let ${\bf C}=(C_0,C_1,...,C_k) \in \Lambda^{(k+1)}_{n\times n}$ and $\x =(x_{0},x_{1},...,x_{k})\in\text{\rm SOL}({\bf C},\d,q)$. Then $\x$ satisfies the following system $$C_0x_{0}=q+\sum_{i=1}^{k} C_ix_{i}~\text{and}~x_{0}\wedge x_{j}=0~\forall~j\in [k].$$  
\end{lemma}
\begin{proof}
As $x_{0}\geq0$, there exists an index set $\alpha \subseteq [n]$ such that $(x_{0})_i=\begin{cases}
	>0 & i\in \alpha\\
	0 & i\in [n]\setminus \alpha
\end{cases}.$  Since $x_{0} \wedge x_{1}=0$, we have $(x_{1})_i=0$ for all $i\in \alpha$. From  $(d_{1}-x_{1})\wedge x_{2}=0$, we get $(d_{1})_i (x_{2})_i =0~\forall i\in \alpha$. This gives that $(x_{2})_i =0~\forall i\in \alpha$. By substituting $(x_{2})_i =0~\forall i\in \alpha$ in $(d_{2}-x_{2})\wedge x_{3}=0$, we obtain $(x_{3})_i =0~\forall i\in \alpha $. Continue the process in the similar way, one can get $(x_{4})_i =(x_{5})_i=...=(x_{k})_i=0~\forall i\in \alpha$. So,  $x_{0}\wedge x_{j}=0~ \forall ~j \in[k]$. This completes the proof.  
\end{proof}
We now prove the boundedness of the solution set of EHLCP when the involved set of matrices has the $R_0$-$W$ property.
\begin{theorem}\label{R_0}
Let ${\bf C}=(C_0,C_1,...,C_k) \in \Lambda^{(k+1)}_{n\times n}$. If ${\bf C}$ has the  $R_0$-$W$ property then  $\text{\rm SOL}({\bf C},\d,q)$ is bounded for every $q\in\mathbb{R}^n$ and $\d \in  \Lambda^{(k-1)}_{n,++}$.
\end{theorem}
\begin{proof} Suppose there exist $q\in\mathbb{R}^n$ and $\d=(d_{1}, d_{2},...,d_{k-1})\in  \Lambda^{(k-1)}_{n,++}$ such that $\text{SOL}({\bf C},\d,q)$ is unbounded. Then there exists a sequence ${\bf x}^{(m)}=( {x^{(m)}_{0}},{x^{(m)}_{1}},...,{x^{(m)}_{k}})$ in $\Lambda^{(k+1)}_n$ such that $||{\bf x}^{(m)}||  \to \infty $ as $m\to \infty$ and it satisfies
\begin{equation}\label{bound}
	\begin{aligned}
		~~~~& C_0 {x^{(m)}_{0}} =q+\sum_{i=1}^{k} C_i {x^{(m)}_{i}} \\ 
		~~~~& {x^{(m)}_{0}} \wedge {x^{(m)}_{1}}=0 ~~\text{and}~~ (d_{j}-{x^{(m)}_{j}})\wedge {x^{(m)}_{j+1}}=0 ~\forall j\in[k-1].
	\end{aligned}
\end{equation}
From the Lemma \ref{l1}, equation \ref{bound} gives that
\begin{equation}\label{bd1}
	\begin{aligned}
		C_0 {x^{(m)}_{0}} =&q+\sum_{i=1}^{k} C_i {x^{(m)}_{i}}  ~~\text{and}~~
		{x^{(m)}_{0}}  \wedge {x^{(m)}_{j}} =&0 ~\forall j\in[k].\\
	\end{aligned}
\end{equation} 
As $\dfrac{{\bf x}^{(m)}}{\|{\bf x}^{(m)}\|}$ is a unit vector for all $m$,  $\dfrac{{\bf x}^{(m)}}{\|{\bf x}^{(m)}\|}$ converges to some vector $\y=(  y_{0},y_{1},...,y_{k}) \in \Lambda^{(k+1)}_{n}$ with $||\y||=1$. 
Now first divide the equation \ref{bd1} by $\|{\bf x}^{(m)}\|$ and then take the limit $m\rightarrow \infty$,  we get 	$$   C_0 y_{0}=\sum_{i=1}^{k} C_i y_{i} ~~\text{and}~~ y_{0}\wedge y_{j}=0 ~\forall j\in[k].$$ This implies that $\y$ must be a zero vector as ${\bf C}$ has the $R_0$-$W$ property,  which contradicts the fact that $||\y||=1$. Therefore $\text{SOL}({\bf C},\d,q)$ is bounded.
\end{proof}
\subsection{Degree of EHLCP}
Let ${\bf C}=(C_0,C_1,...,C_k)\in \Lambda^{(k+1)}_{n \times n}$ and $\d=(d_{1}, d_{2}, ...., d_{k-1})\in \Lambda^{(k-1)}_{n,++}$. We define a function $F:\Lambda^{(k+1)}_n \to \Lambda^{(k+1)}_n $ as \begin{equation}\label{e1}
\begin{aligned}
	F(\x)=\begin{bmatrix}
		C_0 x_{0} -\sum_{i=1}^{k} C_ix_{i}\\ x_{0}\wedge x_{1}\\ (d_{1}-x_{1})\wedge x_{2}\\ (d_{2}-x_{2})\wedge x_{3}\\ .\\ .\\ .\\ (d_{k-1}-x_{k-1})\wedge x_{k}\\
	\end{bmatrix}.\end{aligned}\end{equation}
We denote the degree of $F$ with respect to ${\bf 0}$ over bounded open set $\Omega \subseteq  \Lambda^{(k+1)}_n$ as $\rm{deg}({\bf C},\Omega,{\bf 0})$. It is noted that if ${\bf C}$ has the  $R_0$-$W$ property, in view of the Lemma \ref{l1}, $F(\x)={\bf 0} \Leftrightarrow {\x}={\bf 0}$ which implies that $\text{deg}({\bf C},\Omega,{\bf 0})=\text{deg} ({\bf C},{\bf 0})$ for any bounded open set $\Omega$ contains the origin in $\Lambda^{(k+1)}_n$. We call this degree as EHLCP-degree of ${\bf C}.$ 

We now prove an existence result for EHLCP.
\begin{theorem}\label{P2}
Let ${\bf C}=(C_0,C_1,...,C_k)\in \Lambda^{(k+1)}_{n\times n}$. Suppose the following hold:
\begin{itemize}
	\item[\rm (i)] ${\bf C}$ has the $R_0$-$W$ property.
	\item[\rm (ii)] ${\rm{deg}}({\bf C},{\bf 0})\neq 0$.
\end{itemize}
Then {\rm EHLCP(${\bf C},\d,q$)} has non-empty compact solution for all $q\in\Rr$ and $\d \in \Lambda^{(k-1)}_{n, ++}$.
\end{theorem}
\begin{proof}
As the solution set of EHLCP is closed, it is enough to prove that the solution set is non-empty and bounded.  We first define a homotopy $\Phi: \Lambda_n^{(k+1)} \times [0,1] \to \Lambda_n^{(k+1)}$  as  $$\Phi(\x,s)=\begin{bmatrix}
	
	C_0 x_{0} -\sum_{i=1}^{k} C_ix_{i}-sq\\ x_{0}\wedge x_{1}\\ (d_{1}-x_{1})\wedge x_{2}\\ (d_{2}-x_{2})\wedge x_{3}\\ .\\ .\\ .\\ (d_{k-1}-x_{k-1})\wedge x_{k}\\
\end{bmatrix}.$$ Then, $$\Phi(\x,0)=F(\x)~~\text{and}~\Phi(\x,1)=F(\x)- \hat{q}, \text{where}~~\hat{q}=(q,0,0,...0)\in \Lambda^{(k+1)}_n.$$
By using the similar argument as in above Theorem $\ref{R_0}$, we can easily show that the zero set of homotopy, $X=\{\x:\Phi(\x,s)={\bf 0}~\text{for some}~s\in[0,1]\}$ is bounded. From the property of degree (D2), we get 
$\text{deg}(F,\Omega,{\bf 0})=\text{deg}(F-\hat{q},\Omega,{\bf 0})$ for any open bounded set $\Omega$ containing $X$. As $\text{deg}(F, \Omega,{\bf 0})=\text{deg}({\bf C},{\bf 0})\neq 0$, we obtain $\text{deg}(F-\hat{q},\Omega,{\bf 0})\neq 0$ which implies $\text{SOL}({\bf C},\d,q) $ is non-empty. As ${\bf C}$ has the $R_0$-$W$ property, by Theorem \ref{bound}, $\text{SOL}({\bf C},\d,q) $ is bounded. This completes the proof. 
\end{proof}
\section{ ${\bf SSM}$-$W$ property}
In this section, we first define the ${\bf SSM}$-$W$ {\it property} for the set of matrices which is a generalization of the SSM matrix in the LCP theory, and we then prove that the existence and uniqueness result for the EHLCP when the involved set of matrices have the ${\bf SSM}$-$W$ property.

We now recall that an $n\times n$  real matrix $M$ is called strictly semimonotone (SSM) matrix if [$x\in \Rr_+,~x*Mx\leq 0\Rightarrow x=0$]. We generalize this concept to the set of matrices.
\begin{definition}\rm 
We say that ${\bf C}=(C_0,C_1,...,C_k)\in \Lambda^{(k+1)}_{n\times n}$ has the ${\bf SSM}$-$W$ property if 
\begin{equation*}
	\{C_0x_{0}=\sum_{i=1}^{k} C_ix_{i},~x_{i}\geq 0~\text{and}~~ x_{0}* x_{i}\leq 0~~\forall i\in[k]\}\Rightarrow {\x}=(x_0,x_1,..,x_k)={\bf 0}.
\end{equation*}
\end{definition}
We prove the following result.
\begin{proposition}\label{P2}
Let ${\bf C}=(C_0,C_1,...,C_k)\in \Lambda^{(k+1)}_{n\times n}$. If ${\bf C}$ has the ${\bf SSM}$-$ W$ property, then the followings hold: 
\begin{itemize}
	\item[\rm (i)] $C_0^{-1}$ exists and $C_0^{-1}C_i$ is a strict semimonotone matrix for all $i\in[k].$
	\item[\rm(ii)] $(I,C_0^{-1} C_1,...,C_0^{-1}C_k)$ has the  ${\bf SSM}$-$W$ property.
	\item [\rm(iii)] $(P^TC_0P,P^TC_1P,...,P^TC_kP)$ has the ${\bf SSM}$-$ W$ property for any permutation matrix $P$ of order $n$.
\end{itemize} 
\end{proposition}
\begin{proof}
(i): Suppose there exists a vector $x_{0}\in \Rr$ such that $C_0x_{0}=0$. Then we have $$C_0x_{0}=C_1 0+C_2 0+ ...+C_k 0.$$ This gives that $x_{0}=0$  as ${\bf C}$ has the  ${\bf SSM}$-$ W$ property. Thus $C_0$ is invertible.

Now we prove the second part of (i). Without loss of generality, it is enough to prove that $C^{-1}_0 C_1$ is a strictly semimonotone matrix. Suppose there exists a vector $y \in\Rr$ such that $y \geq 0$ and $y * (C_0^{-1}C_1) y \leq 0.$  Let $y_0:=(C_0^{-1}C_1)y$, $y_1:=y$ and $y_i:=0$ for all $2\leq i\leq k$.  Then we get $$C_0y_{0}=C_1 y_{1}+C_2 y_{2}+...+C_iy_{i}+..+C_k y_{k},~~y_{j} \geq 0~\text{and}~y_{0} * y_{j}\leq 0~\forall j\in [k].$$ 
Since ${\bf C}$ has the  ${\bf SSM}$-$ W$ property, $y_{j}=0~~\forall j\in [k]$. Thus $C_0^{-1}C_1$ is a strict semimonotone matrix. This completes the proof.

(ii): It follows from the definition of the  ${\bf SSM}$-$W$ property.

(iii): Let ${\bf x}=(x_{0},x_{1},...,x_{k})\in \Lambda^{(k+1)}_n$ such that 
$$ (P^TC_0P )x_{0}=\sum_{i=1}^{k} (P^TC_iP) x_{i},~x_{j}\geq 0~\text{and}~x_{0}* x_{j}\leq 0 ~\forall j\in[k].$$ As $P$ is a non-negative matrix and  $PP^T=P^TP$, we can rewrite the above equation as
$$ C_0Px_{0}=\sum_{i=1}^{k} C_iP x_{i},~ P x_{j} \geq 0 ~\text{and}~Px_{0}* Px_{j}\leq 0 ~\forall j\in[k].$$
By the  ${\bf SSM}$-$ W$ property of ${\bf C}$, $P x_{j}=0$ for all $0\leq j \leq k$ which implies $\x ={\bf 0}$. This completes the proof.
\end{proof} 
In the above Proposition \ref{P2}, it can be seen easily that the converse of the item (ii) and (iii) are valid. But the converse of item (i) need not be true. The following example illustrates this.
\begin{example}\rm
Let ${\bf C}=(C_0,C_1,C_2)\in \Lambda^{(3)}_{2\times 2}$,  where $$C_0=\begin{bmatrix}
	1&0\\0&1\\
\end{bmatrix},~C_1=\begin{bmatrix}
	1&-2\\0&1\\
\end{bmatrix},C_2=\begin{bmatrix}
	1&0\\-2&1\\
\end{bmatrix}.$$ It is easy to check that $C_0^{-1}C_1=C_1$ and $C_0^{-1}C_2=C_2$ are $P$  matrix. So, $C_0^{-1}C_1$ and $ C_0^{-1}C_2$ are SSM matrix. Let $\x=(x_{0},x_{1},x_{2})=((0,0)^T,(1,1)^T,(1,1)^T)\in \Lambda^{(3)}_2$. Then we can see that the non-zero $\bf{x}$ satisfies $$C_0x_{0}=C_1x_{1}+C_2x_{2},~x_{1}\geq 0,~x_{2} \geq 0~\text{and}~x_{0}*x_{1}=0=x_{0} * x_{2}.$$ So ${\bf C}$ can not have the ${\bf SSM}$-$ W$ property.
\end{example}
The following result is a generalization of a well-known result in matrix theory that every $P$ matrix is a SSM matrix.
\begin{theorem}\label{T4}
Let  ${\bf C}=(C_0,C_1,...,C_k)\in \Lambda^{(k+1)}_{n\times n}$. If  ${\bf C}$ has the column $W$-property, then ${\bf C}$ has the ${\bf SSM}$-$ W$ property. 
\end{theorem}
\begin{proof}
Suppose there exists a non-zero vector ${\bf x}=(x_{0},...,x_{k})\in \Lambda^{(k+1)}_n$ such that $$C_0x_{0}=\sum_{i=1}^{k} C_ix_{i},~x_{j}\geq 0,~ x_{0}* x_{j}\leq 0~\forall j\in[k].$$ 
Consider a vector $y\in\Rr$ whose $j^{\rm{th}}$ component is given by $$y_j=\begin{cases}
	-1 & \text{if} ~(x_{0})_j>0\\ 1 & \text{if}  ~(x_{0})_j<0\\1 & \text{if} ~(x_{0})_j=0 ~\text{and}~(x_{i})_j\neq 0~\text{for some}~ i\in[k]\\0 & \text{if}~ (x_{0})_j=0 ~\text{and} ~(x_{i})_j= 0 ~\text{for all}~i\in[k] 
\end{cases}.$$
As ${\bf{x}}$ is a non-zero vector, ${\bf{y}}$ must be a non-zero vector. Consider the diagonal matrices $D_{0}, D_{1},...,D_{k}$ which are defined by $$(D_{0})_{jj}=\begin{cases}(x_{0})_j & \text{if} ~(x_{0})_j>0\\ -(x_{0})_j & \text{if}  ~(x_{0})_j<0\\0 & \text{if} (x_{0})_j=0~\text{and}~(x_{i})_j\neq 0~\text{for some}~ i\in[k] \\1 & \text{if}~ (x_{0})_j=0 ~\text{and} ~(x_{i})_j= 0 ~\text{for all}~i\in[k]
\end{cases}$$ and  for all $i\in [k]$,
$$(D_{i})_{jj}=\begin{cases}
	0&\text{if } (x_{0})_j>0\\ (x_{i})_j&\text{ else }
\end{cases}.$$
It is easy to verify that $D_{0}, D_{1},...,D_{k}$ are non-negative diagonal matrices and  $\text{diag}(D_{0}+ D_{1}+...+D_{k})>0$. And also note that  
\begin{equation}\label{22}
	x_{0}=-D_{0}y~\text{and}~x_{i}=D_{i}y~\forall i\in [k].
\end{equation}
By substituting the Equation \ref{22} in $C_0x_{0}=\sum_{i=1}^{k} C_ix_{i}$, we get \begin{equation*}
	C_0(-D_{0}y)=\sum_{i=1}^{k} C_iD_{i}(y) \Rightarrow 
	\big(C_0D_{0}+C_1D_{1}+...+C_kD_{k}\big)y=0.
\end{equation*} 
This implies that det$(C_0D_{0}+C_1D_{1}+...+C_kD_{k}\big)=0$. So, ${\bf C}$ does not have the column $W$-property from Theorem \ref{P1}. Thus we get a contradiction. Therefore, ${\bf C}$ has the ${\bf SSM}$-$W$ property.  
\end{proof}
The following example illustrates that the converse of the above theorem is invalid. 
\begin{example}\rm
Let ${\bf C}=(C_0,C_1,C_2) \in \Lambda^{(3)}_{2\times 2}$ such that $$C_0=\begin{bmatrix}
	1&0\\0&1\\
\end{bmatrix},~C_1=\begin{bmatrix}
	1&1\\1&1\\
\end{bmatrix},C_2=\begin{bmatrix}
	1&1\\1&1\\
\end{bmatrix}.$$  Suppose $ {\bf{w}} =(x,y,z)\in \Lambda^{3}_2$ such that $$ C_{0}x=C_{1}y+C_{2}z~\text{and}~ y,z\geq 0, x*y\leq 0, x*z\leq 0.$$ 
From $C_{0}x=C_{1}y+C_{2}z$, we get
$$\begin{bmatrix}
	x_{1}\\x_{2}
\end{bmatrix}=\begin{bmatrix}
	y_1+y_2+z_1+z_2\\y_1+y_2+z_1+z_2\\
\end{bmatrix}.$$ As $x*y\leq 0,~x*z\leq 0$ and from the above equation, we have 
\begin{equation}\label{33}
	\begin{aligned}
		y_1&(y_1+y_2+z_1+z_2)\leq 0~\text{and}~ y_2(y_1+y_2+z_1+z_2)\leq 0,\\
		z_1&(y_1+y_2+z_1+z_2)\leq 0~\text{and}~
		z_2(y_1+y_2+z_1+z_2)\leq 0.\\
	\end{aligned}
\end{equation}
Since $y,z\geq 0$, from the equation \ref{33}, we get $x=y=z=0.$ Hence ${\bf C}$ has the ${\bf SSM}$-$ W$ property. As $\text{det}(C_1)=0$, by  the definition of the column $W$-property,  ${\bf C}$ does not have the column $W$-property. 
\end{example} 
We now give a characterization for ${\bf SSM}$-$W$ property. 
\begin{theorem}\label{CW}
Let ${\bf C}=(C_0,C_1,...,C_k)\in \Lambda^{(k+1)}_{n\times n}$ has the ${\bf SSM}$-$ W$ property if and only if $(C_0,C_1D_{1}+C_2D_{2}+...+C_kD_{k})\in \Lambda_{n\times n}^{(2)}$ has the ${\bf SSM}$-$ W$ property for any set of non-negative diagonal matrix $(D_{1},D_{2},...,D_{k})\in \Lambda^{(k)}_{n\times n}$ with $\text{\rm diag}(D_{1}+D_{2}+...+D_{k})>0$.
\end{theorem}
\begin{proof}
{\it Necessary part}: Let $(D_{1},D_{2}...,D_{k})\in \Lambda^{(k)}_{n\times n}$ be the set of non-negative diagonal matrix with $\text{\rm diag}(D_{1}+D_{2}+...+D_{k})>0$. Suppose there exist vectors $x_{0}\in\mathbb{R}^n$ and $y\in \mathbb{R}^n_+ $ such that $$C_0x_{0}=\big(C_1D_{1}+C_2D_{2}+...+C_kD_{k}\big)y~~\text{and}~~x_{0}*y\leq 0.$$
For each  $i\in[k]$, we set $x_{i}:=D_{i}y$. As  each $D_{i}$ is a non-negative diagonal matrix, from $x_{0}*y\leq 0$, we get $x_{0}*x_{i}\leq 0~ \forall i\in[k]$. Then we have $$C_0x_{0}=C_1x_{1}+C_2x_{2}+...+C_kx_{k},$$ $$x_{i}\geq 0,~x_{0}*x_{i}\leq 0~\forall i\in[k].$$  As ${\bf C}$ has the ${\bf SSM}$-$ W$ property of ${\bf C}$, we must have $x_{0}=x_{1}=...=x_{k}=0.$ This implies $x_{1}+x_{2}+...+x_{k}=(D_{1}+D_{2}+...+D_{k})y=0$. As $\text{\rm diag}(D_{1}+D_{2}+....+D_{k})>0$, we have $y=0$. This completes the necessary part.

\noindent{\it Sufficiency part}: Let ${\bf x}=(x_{0},x_{1},...,x_{k})\in \Lambda^{(k+1)}_n$ such that  \begin{equation} \label{SS}
	C_{0}x_{0}=C_1x_{1}+C_2x_{2}+...+C_kx_{k}~~\text{and}~~x_{j}\geq 0,~x_{0}*x_{j}\leq 0~\forall j\in[k].
\end{equation}

We now consider an $n\times k$ matrix $X$ whose $j^{\rm th}$ column as $x_{j}$ for $j\in[k]$. So, $X=[x_{1}~x_{2}~...~x_{k}]$. Let $S:=\{i\in [k]: i^{\rm{th}} ~\text{row sum of $X$ is zero}\}$. From this,  we define a vector $y\in\Rr$ and diagonal matrices $D_{1},D_{2},..,D_{k}$ such that
$$y_i=\begin{cases}
	1 & i\notin S\\ 0 & i\in S\\
\end{cases}~~\text{and}~~ (D_{j})_{ii}=\begin{cases}
	(x_{j})_i & i\notin S\\~ 1 &i\in S\\
\end{cases},$$ where $(D_{j})_{ii}$ is the diagonal entry of $D_{j}$ for all $j\in[k]$.
It can be seen easily that $D_{j}y=x_{j}$ for all $j\in[k]$ and each $D_{j}$ is a non-negative diagonal matrix with $\text{\rm diag}(D_{1}+D_{2}+...+D_{k})>0$. Therefore, from  equation \ref{SS}, we get
$$C_0x_{0}=\big(C_1D_{1}+C_2D_{2}+...+C_kD_{k}\big)y,$$
$$x_{0}*y\leq 0.$$ From the hypothesis, we get $x_{0}=0=y$ which implies ${\bf{x}}={\bf 0}$. This completes the sufficiency part.
\end{proof} We now give a  characterization for the column $W$-property.
\begin{theorem}\label{CD}
Let ${\bf C}=(C_0,C_1,...,C_k)\in \Lambda^{(k+1)}_{n\times n}$ has the column-$ W$ property if and only if $(C_0,C_1D_{1}+C_2D_{2}+...+C_kD_{k})\in \Lambda_{n\times n}^{(2)}$ has the column-$ W$ property for any set of non-negative diagonal matrices $D_{1},D_{2},...,D_{k}$ of order $n$ with ${\rm{diag}}(D_{1}+D_{2}+...+D_{k})>0$.
\end{theorem}
\begin{proof}  
{\it Necessary part}: It is obvious. 

\noindent{\it Sufficiency part}: Let $\{E^{0},E^{1},...,E^{k}\}$ be a set of  non-negative diagonal matrices of order $n$ such that ${\rm{diag}}(E^{0}+E^{1}+...+E^{k})>0$. We claim that $\det(C_0E^{0}+C_1E^{1}+...+C_kE^{k}) \neq 0$.

To prove this, we first construct a set of non-negative diagonal matrices $D_{1},D_{2},...,D_{k}$ and $E$ as follows: $$(D_{j})_{ii}=\begin{cases}
	E^{j}_{ii}&\text{~if~} \sum_{m=1}^{k}E^{m}_{ii}\neq0\\
	1&\text{~if~} \sum_{m=1}^{k}E^{m}_{ii}=0\\
\end{cases}\text{ and } E_{ii}=\begin{cases}
	1 &\text{~if~} \sum_{m=1}^{k}E^{m}_{ii}\neq0\\
	0&\text{~if~} \sum_{m=1}^{k}E^{m}_{ii}=0\\
\end{cases},$$ where $(D_{j})_{ii}$ is $ii^{\rm th}$ diagonal entry of $D_{j}$ for $j\in[k]$ and $ E_{ii}$ is $ii^{\rm th}$ diagonal entry of matrix $E$. By an easy computation, we have $D_{j}E=E^{j}~\forall j\in[k]$ and ${\rm diag}(D_{1}+D_{2}+...+D_{k})>0$. From ${\rm{diag}}(E^{0}+E^{1}+...+E^{k})>0$, we get ${\rm diag}(E^{0}+E)>0$. As $D_{j}E=E^{j}~\forall j\in[k]$ and $(C_0,C_1D_{1}+C_2D_{2}+...+C_kD_{k})$ has column $W$-property, by Theorem \ref{P1}, we have  \begin{equation*}
	\begin{aligned}
		\det(C_0E^{0}+C_1E^{1}+...+C_kE^{k})&=\det(C_0E^{0}+C_1D_{1}E+...+C_kD_{k}E)\\
		&=\det(C_0E^{0}+(C_1D_{1}+...+C_kD_{k})E) \neq 0.
	\end{aligned}
\end{equation*}
Hence ${\bf C}$ has the column $W$-property. This completes the proof.
\end{proof}
A well-known result in the standard LCP is that strictly semimonotone matrix and $P$  matrix are equivalent in the class of $Z$ matrices (see, Theorem 3.11.10 in \cite{LCP}). Analogue this result, we prove the following theorem.
\begin{theorem}\label{cssm}
Let ${\bf C}=(C_0,C_1,...,C_k)\in \Lambda_{n\times n}^{(k+1)}$ such that $C_0^{-1}C_i$ be a $Z$ matrix for all $i\in[k]$. Then the following statements are equivalent.\begin{itemize}
	\item [\rm (i)] ${\bf C}$ has the column $W$-property.
	\item[\rm(ii)]  ${\bf C}$ has the ${\bf SSM}$-$ W$ property.
\end{itemize}
\end{theorem}
\begin{proof}
(i)$\implies$(ii): It follows from Theorem \ref{T4}.

(ii)$\implies$(i): Let $\{D_{1}, D_{2},...,D_{k}\}$  be the set of non-negative diagonal matrices of order $n$ such that $\text{\rm diag}(D_{1}+D_{2}+...+D_{k})>0$. In view of Theorem \ref{CD}, it is enough to prove that $(C_0,C_1D_{1}+C_2D_{2}+...+C_kD_{k})$ has the column $W$-property.

As  ${\bf C}$ has the ${\bf SSM}$-$ W$ property, by Theorem \ref{CW}, we have $(C_0,C_1D_{1}+...+C_kD_{k})$ has  the ${\bf SSM}$-$ W$ property. So, by Proposition \ref{P2}, $\big(I,C_0^{-1}\big(C_1D_{1}+...+C_kD_{k}\big)\big)$ has the ${\bf SSM}$-$ W$ property and $C_0^{-1}\big(C_1D_{1}+C_2D_{2}+...+C_kD_{k}\big)$ is a strict semimonotone matrix. As $C_0^{-1}C_i$ is a $Z$ matrix, we get $C_0^{-1}\big(C_1D_{1}+C_2D_{2}+...+C_kD_{k}\big)$ is also a $Z$ matrix. Hence $C_0^{-1}\big(C_1D_{1}+C_2D_{2}+...+C_kD_{k}\big)$ is a $P$  matrix. So, by Theorem \ref{C1},  $(C_0,C_1D_{1}+C_2D_{2}+...+C_kD_{k})$ has the column $W$-property. Hence we have our claim.
\end{proof}
\begin{corollary}
Let ${\bf C}=(C_0,C_1,...,C_k)\in \Lambda_{n\times n}^{k+1}$ such that $C_0^{-1}C_i$ be a $Z$ matrix for all $i\in[k]$. Then the following statements are equivalent.
\begin{itemize}
	\item [\rm (i)]  ${\bf C}$ has the ${\bf SSM}$-$ W$ property.
	\item [\rm(ii)] For all $q\in\Rr$ and $\d\in\Lambda^{(k-1)}_{n,++}$, {\rm EHLCP}$({\bf C},\d,q)$ has a unique solution.
\end{itemize}
\end{corollary}
\begin{proof}
(i) $\implies$(ii): It follows from Theorem \ref{cssm} and Theorem \ref{P1}.
(ii)$\implies$(i): It follows from Theorem \ref{P1} and Theorem \ref{T4}.
\end{proof}
In the standard LCP \cite{deg}, the strictly semimonotone matrix gives the existence of a solution of LCP. We now prove that the same result holds in EHLCP. 
\begin{theorem}\label{degg}
Let ${\bf C}=(C_0,C_1,...,C_k)\in \Lambda^{(k+1)}_{n\times n}$ has the ${\bf SSM}$-$W$ property, then $\rm{{SOL}}({\bf C},\d,q)\neq \emptyset$ for all $q\in\Rr$ and $\d\in \Lambda^{(k+1)}_{n,++}$. 
\end{theorem}
\begin{proof}
As ${\bf C}$ has the ${\bf SSM}$-$W$ property$, {\bf C}$ has the $R_0$-$W$ property. From Theorem \ref{P2}, it is enough to prove that ${\rm deg}({\bf C},{\bf 0})\neq 0$. To prove this, 
we consider a homotopy $\Phi: \Lambda_n^{(k+1)} \times [0,1] \to \Lambda_n^{(k+1)}$ as $$\Phi({\bf x},t)=t\begin{bmatrix}
	C_0x_{0} \\ x_{1}\\x_{2}\\x_{3}\\.\\.\\.\\x_{k} \\	
\end{bmatrix}+(1-t)\begin{bmatrix}
	C_0 x_{0} -\sum_{i=1}^{k} C_ix_{i}\\ x_{0}\wedge x_{1}\\ (d_{1}-x_{1})\wedge x_{2}\\ (d_{2}-x_{2})\wedge x_{3}\\ .\\ .\\ .\\ (d_{k-1}-x_{k-1})\wedge x_{k}\\
\end{bmatrix}.$$
Let $F({\bf x}):=\Phi({\bf x},0)$ and $G({\bf x}):=\Phi({\bf x},1)$. 
We first prove that the zero set $X=\{\x:\Phi(\x,t)={\bf 0}~\text{for some}~t\in[0,1]\}$ of homotopy $\Phi$ contains only zero.  We consider the following cases.

{\it Case 1}: Suppose $t=0$  or $t=1$. If $t=0$, then $\Phi({\bf x},0)={\bf 0}\implies F(\x)={\bf 0}$.  As $\bf{C}$ has the ${\bf SSM}$-$W$ property, by Lemma \ref{l1}, we have 
$F({\bf x})={\bf 0}\Rightarrow {\bf x}={\bf 0}$. If $t=1$, then $\Phi({\bf x},1)={\bf 0}\implies G(\x)={\bf 0}$. Again by $\bf{C}$ has the ${\bf SSM}$-$W$ property, $C^{-1}_0$ exists, which implies that $G$ is a one-one map. So, $G({\bf x})={\bf 0}\Rightarrow {\bf x}={\bf 0}.$ 

{\it Case 2}: Suppose $t\in(0,1)$. Then $\Phi(\x,t)={\bf 0}$ which gives that
\begin{equation}\label{pp}
	\begin{bmatrix}
		C_0 x_{0} -\sum_{i=1}^{k} C_ix_{i}\\ x_{0}\wedge x_{1}\\ (d_{1}-x_{1})\wedge x_{2}\\ (d_{2}-x_{2})\wedge x_{3}\\ .\\ .\\ .\\ (d_{k-1}-x_{k-1})\wedge x_{k}\\
	\end{bmatrix}=-\alpha\begin{bmatrix}
		C_0x_{0} \\ x_{1}\\x_{2}\\x_{3}\\.\\.\\.\\x_{k} \\	
	\end{bmatrix},~~\text{where}~~\alpha=\dfrac{t}{1-t}>0.
\end{equation}
From the second row of above equation, we have $$x_{0}\wedge  x_{1}=-\alpha {x_{1}}\implies \text{min}\{{x_{0}}+\alpha {x_{1}} ,(1+\alpha){x_{1}}\}=0.$$ By Proposition \ref{star}, we get $x_{1} \geq 0$ and $({x_{0}}+\alpha {x_{1}}) * (1+\alpha){x_{1}} =0 $ which implies that ${x_{0}} * {x_{1}} \leq 0.$ Set $\Delta:=\{i\in [n]: ({x_{1}})_{i} >0\}$. So, we have \begin{equation}\label{EX}
	(x_{0})_{i}=\begin{cases}
		~\leq 0~&{\text{if}}~ i\in \Delta\\
		~\geq 0~& {\text{if}} ~i \notin \Delta
	\end{cases}~~~\text{and}~~(x_{1})_{i}=\begin{cases}
		>0~~\text{if}~i\in \Delta\\
		=0~~\text{if}~i\notin \Delta
	\end{cases}.
\end{equation}
From third row of the equation \ref{pp}, we have $ (d_{1}-x_{1})\wedge  x_{2}=-\alpha {x_{2}}$ which is equivalent $$\min\{d_{1}-x_{1}+\alpha {x_{2}}, (1+\alpha){x_{2}}\}=0.$$ This gives that $x_{2}\geq 0$ and $(d_{1}-x_{1}+\alpha {x_{2}})*(1+\alpha){x_{2}}=0$. As $d_{1}>0$ and from the last term in equation \ref{EX}, we have $$(x_{2})_{i}=\begin{cases}
	\geq 0~\text{if}~i\in \Delta\\
	=0 ~\text{if}~i\notin \Delta
\end{cases}. $$
This leads that $x_{0} * x_{2}\leq 0$. By continuing the similar argument for the remaining rows, we get $$x_{j}\geq 0~~\text{and}~x_{0}*x_{j}\leq 0~\forall j\in [k].$$
From the first row of the equation \ref{pp}, the vectors ${\bf{x}}=(x_{0},x_{1},...,x_{k})$ satisfies $$	C_0(1+\alpha)  x_{0} =\sum_{i=1}^{k} C_i x_{i}~~\text{and}~~x_{j}\geq0,~{x_{0}}*{x_{j}}\leq 0, ~j\in[k].$$ So, ${\x}={\bf 0}$ as ${\bf C}$ has the ${\bf SSM}$-$ W$ property. 

From both cases, we get $X$ contains only zero. By the homotopy invariance property of degree (D2), we have $\text{deg}(\Phi(\x,0),\Omega,{\bf 0})=\text{deg}\big(\Phi(\x,1),\Omega,{\bf 0}\big)$ for any bounded open set containing ${\bf 0}$. As  $G$ is a continuous one-one function, by Proposition \ref{ND},  we have  $$\text{deg}\big({\bf C},{\bf 0}\big)=\text{deg}\big(\Phi(\x,0),\Omega,{\bf 0}\big)= \text{deg}\big(F,\Omega,{\bf 0}\big)=\text{deg}\big(G,\Omega,{\bf 0}\big)\neq 0.$$ This completes the proof.
\end{proof}
We now recall that  a matrix $A \in\R$ is said to be a $M$ matrix if it is $Z$ matrix and $A^{-1}(\mathbb{R}^n_+)\subseteq \mathbb{R}^n_+.$ We prove a uniqueness  result for EHLCP when $q\geq0$ and $\d\in\Lambda^{(k-1)}_{n,++}.$    
\begin{theorem}\label{smq}
Let ${\bf C}=(C_0,C_1,...,C_k)\in \Lambda^{(k+1)}_{n\times n}$ has the ${\bf SSM}$-$ W$ property. If $C_0$ is a $M$ matrix. then for every $q \in \Rr_+$ and for every $\d \in  \Lambda^{(k-1)}_{n,++}$, $\text{\rm EHLCP}({\bf C},\d,q)$ has a unique solution.
\end{theorem}
\begin{proof}
Let  $q \in \Rr_+$ and $\d=(d_{1}, d_{2},...,d_{k-1}) \in  \Lambda^{(k-1)}_{n,++}$. We first show $(C_0^{-1}q,0,...,0)\in\text{SOL}({\bf C},\d,q).$ As $C_0$ is a $M$ matrix and $q \in \Rr_+$, we have $C_0^{-1}q\geq 0$.  If we set ${\bf{y}}=(y_{0},y_{1},...,y_{k}):=(C_0^{-1}q,0,...,0)\in \Lambda^{(k+1)}_{n}$, then we can see easily that $(y_{0},y_{1},...,y_{k})$ satisfies that $$C_0 y_{0}=q+\sum_{i=1}^k C_i y_{i},~ y_{0}\wedge y_{1}=0~\text{and}~(d_{j}-y_{j})\wedge y_{j+1}=0~~\forall j\in[k-1].$$ Hence $(C_0^{-1}q,0,...,0)\in \text{SOL} ({\bf C},\d,q).$

Suppose ${\bf x}=(x_{0},x_{1},...,x_{k})\in \Lambda^{(k+1)}_n$ is an another solution to EHLCP($C,q,d$). Then, 
\begin{equation}\label{ssm}
	\begin{aligned}
		C_0 x_{0}=q+\sum_{i=1}^{k} C_ix_{i},~
		x_{0}\wedge x_{1}=0, ~ (d_{j}-x_{j})\wedge x_{j+1}=0~ \forall j\in[k-1].
	\end{aligned}
\end{equation}
From the Lemma \ref{l1}, we have
\begin{equation}\label{unique}
	C_0x_{0}=q+\sum_{i=1}^{k} C_ix_{i}~\text{and}~x_{0}\wedge x_{j}=0~\forall~j\in [k].
\end{equation}
We let ${\bf{z}}:=\x-\y$,  then ${\bf{z}}=(x_{0}-C_0^{-1}q, x_{1},x_{2},..,x_{k})$. By an easy computation, from Equation \ref{unique}, we get
$$C_0 (x_{0}-C_0^{-1}q)=\sum_{i=1}^k C_i x_{i}$$ and 
$$ x_{j}\geq 0,~~(x_{0}-C_0^{-1}q)*x_{j}=x_{0}*x_{j}-C_0^{-1}q*x_{j}=-C_0^{-1}q*x_{j}\leq 0~\forall j\in [k].$$ 
Since ${\bf C}$ has the ${\bf SSM}$-$ W$ property, ${\bf{z}}={\bf 0}$ which implies that $(x_{0},x_{1},...,x_{k})=(C_0^{-1}q,0,...,0).$ This completes the proof.
\end{proof}

\section{Connected solution set and Column ${W_0}$ property }
In this section, we give a necessary and sufficient condition for the connected solution set of the EHLCP.
\begin{definition}\rm
Let ${\bf C}=(C_0,C_1,...,C_k)\in\Lambda^{(k+1)}_{n\times n}$. We say that ${\bf C}$ is  connected if  $\text{SOL}({\bf C},\d,q)$ is connected for all $q\in\Rr$ and for all $\d\in\Lambda^{(k-1)}_{n,++}$.
\end{definition}
We now recall some definitions and results to proceed further.
\begin{definition}\rm\cite{semi}
A subset of $\Rr$ is said to be a semi-algebraic set it  can be represented as, $$S=\displaystyle\bigcup_{u=1}^{s}\bigcap_{v=1}^{r_u}\{\x\in\Rr;f_{u,v}(\x)*_{uv} 0\}, $$
where for all $u\in[s]$ and for all $v\in[r_u]$, $*_{uv}\in\{\ >, =\}$ and $f_{u,v}$ is in the space of all real polynomials.
\end{definition}
\begin{theorem}[\cite{semi}]\label{sctd}
Let $S$ be a semi-algebraic set. Then $S$ is connected iff $S$ is path-connected.
\end{theorem}
\begin{lemma}\label{CC}
The $\text{\rm SOL}({\bf C},\d,q)$ is a semi-algebraic set.
\end{lemma} 
\begin{proof}
It is clear from the definition of $\text{SOL}({\bf C},\d,q)$.
\end{proof}
The following result gives a necessary condition for a connected solution whenever $C_0$ is a $M$ matrix.
\begin{theorem}
Let $C_0\in\R $ be a $M$ matrix. If  ${\bf C}=(C_0,C_1,...,C_k)\in\Lambda^{(k+1)}_{n\times n}$ is connected, then $\text{\rm SOL}({\bf C},\d,q)=\{(C_0^{-1}q,0,...,0)\}$ for all $q\in\mathbb{R}^n_{++}$ and for all $\d\in\Lambda^{(k-1)}_{n,++}$.
\end{theorem}
\begin{proof}
Let $q\in\Rr_{++}$ and $\d=(d_{1}, d_{2},...,d_{k-1})\in\Lambda^{(k-1)}_{n,++}$.  It can be seen from the proof of Theorem \ref{smq} that $\x=(C_0^{-1}q,0,...,0)\in\text{SOL}({\bf C},\d,q)$. We now show that $\x$ is the only solution to EHLCP(${\bf C},\d,q$).

Assume contrary. Suppose $\y$  is another solution to EHLCP(${\bf C},\d,q$). As $\text{SOL}({\bf C},\d,q)$ is connected, by Lemma \ref{CC} and Theorem \ref{sctd}, it is path-connected. So, there exists a path $\gamma=(\gamma^{0},\gamma^{1},...,\gamma^{k}):[0,1]\rightarrow \text{SOL}({\bf C},\d,q)$ such that  $$\gamma(0)=\x, ~\gamma(1)=\y~ \text{and}~ \gamma(t)\neq \x ~\forall t>0.$$ 
Let $\{t_m \} \subseteq(0,1)$ be a sequence such that ${t_m}\to 0$ as $m\rightarrow \infty$. Then, by the continuity of $\gamma$, $\gamma(t_m)\rightarrow \gamma(0)=\x$ as $m\to \infty$. Since $\big(\gamma^{0}(t_m),\gamma^{1}(t_m),...\gamma^{k}(t_m)\big)\in\text{SOL}({\bf C},\d,q),$ \begin{equation*}\begin{aligned}
		C_0\gamma^{0}(t_m) &=q+\sum_{i=1}^{k} C_i\gamma^{i}(t_m),\\ 
		\gamma^{0}(t_m)\wedge\gamma^{1}(t_m)=0& ~\text{and}~\big(d_{j}-\gamma^{j}(t_m)\big)\wedge
		\gamma^{({j+1})}(t_m) =0~\forall j\in[k-1].\\	\end{aligned}
\end{equation*}
Now we claim  that there exists a subsequence \{$t_{m_l}\}$ of $\{t_{m}\}$  such that $$\big(\gamma^{j}(t_{m_l})\big)_i\neq 0 ,\text{ for some }j\in[k] \text{ and for some } i\in[n].$$ Suppose the claim is not true. This means that for given any subsequence $\{t_{ml}\}$ of $\{t_m\}$, there exists $m_0\in \mathbb{N}$ such that for all $m_l \geq m_0$, we have  $$\big(\gamma^{j}(t_{m_l})\big)_i =0~~\forall i\in [n]~~\forall j\in [k].$$ So, $\gamma^{j} (t_m)$ is an eventually zero sequence for all  $j\in[k]$. This implies that there exists a natural number $m_0$ such that $$\gamma^{1}(t_m)=\gamma^{2}(t_m)=...=\gamma^{k}(t_m)=0~~\forall m\geq m_0.$$As $\big(\gamma^{0}(t_m),\gamma^{1}(t_m),...\gamma^{k}(t_m)\big)\in\text{SOL}({\bf C},\d,q)$, we get $\gamma^{0}(t_m)=C^{-1}_0 (q)~~\forall m \geq m_0$. This gives us that $\gamma(t_m)=\x$ for all $m\geq m_0$ which contradicts the fact that  $\gamma(t_m)\neq\x$ for all $m$. Therefore, our claim is true.
No loss of generality, we assume a sequence $\{t_m\}$ itself satisfies the condition $$\big(\gamma^{j}(t_{m})\big)_i\neq 0 ,\text{ for some }j\in[k] \text{ and for some } i\in[n].$$
We now consider the following cases for possibilities of $j$.

{\it Case 1} : If $j=1$, then $(\gamma^{0}(t_{m}))_i(\gamma^{1}(t_{m}))_i=0$ which leads to $(\gamma^{0}(t_{m}))_i=0.$ This implies that $$0=\displaystyle{\lim_{m \to \infty} \gamma^{0}(t_{m})_{i}=(C_0^{-1}q})_i.$$  But $(C_0^{-1}q)>0$ as $C_0$ is a $M$ matrix. This is not possible. So, $j\neq 1$. 

{\it Case 2} : If $2\leq j\leq k $, then we have $(d_{j-1}-\gamma^{j-1}(t_{m}))_i
(\gamma^{j}(t_{m}))_i=0$ which gives that $(d_{j-1}-\gamma^{j-1}(t_{m}))_i=0.$ By taking limit $m\rightarrow\infty$, 
$$0=\lim_{m\rightarrow\infty}(d_{j-1}-\gamma^{j-1}(t_{m}))_i=(d_{j-1})_i-(\gamma^{j-1}(0))_i= (d_{j-1})_i>0.$$ This is not possible. 

From both cases, there is no such a $j$ exists. This contradicts the fact. Hence $\x=(C_0^{-1}q,0,...,0)$ is the only solution to $\text{EHLCP}({\bf C},\d,q)$. 
\end{proof}
The following result gives a sufficient condition for a connected solution to EHLCP.
\begin{theorem}
Let  ${\bf C}:=(C_0,C_1,...,C_k)\in \Lambda^{(k+1)}_{n \times n}$ has the column $W_0$-property. If $\text{\rm SOL}({\bf C},\d,q)$ has a bounded connected component, then $\text{\rm SOL}({\bf C},\d,q)$ is connected.
\end{theorem}
\begin{proof} If SOL$({\bf C},\d,q)= \emptyset$, then we have nothing to prove.
Let SOL$({\bf C},\d,q)\neq \emptyset$ and $A$ be a connected component of SOL$({\bf C},\d,q)$. If SOL$({\bf C},\d,q)=A$, then we are done. Suppose SOL$({\bf C},\d,q)\neq A.$ Then there exists $\y=(y_{0},y_{1},..,y_{k})\in \text{SOL}({\bf C},\d,q)\setminus A$.
As $A$ is a bounded connected component of SOL$({\bf C},\d,q)$,  we can find an open bounded set $\Omega \subseteq \Lambda^{(k+1)}_{n}$ which contains $A$ and it does not intersect with other component of $\text{SOL}({\bf C},\d,q)$. Therefore $\y \notin\Omega$ and $\partial{(\Omega)}\cap\text{SOL}({\bf C},\d,q)=\emptyset.$ Since  ${\bf C}$ has the column $W_0$-property, there exists ${\bf N}:=(N_0,N_1,...,N_k)\in \Lambda^{(k+1)}_{n\times n}$ such that  ${\bf C+\epsilon N}:=(C_0+\epsilon N_0,C_1+\epsilon N_1,...,C_k+\epsilon N_k)$ has the column $W$-property for every $\epsilon>0$. 

Let ${\bf z}=(z_{0},z_{1},...,z_{k})\in A$ and  $\epsilon >0$, we define functions $H_1$, $H_2$ and $H_3$ as follows:
$$H_1 (\x)=\begin{bmatrix}
	C_0 x_{0} -\sum_{i=1}^{k} C_ix_{i}-q\\ x_{0}\wedge x_{1}\\ (d_{1}-x_{1})\wedge x_{2}\\ .\\ .\\ .\\ (d_{k-1}-x_{k-1})\wedge x_{k}\\
\end{bmatrix},$$ $$H_2 (\x)=\begin{bmatrix}
	(C_0+\epsilon N_0) x_{0} -\sum_{i=1}^{k} (C_i+\epsilon N_i)x_{i}+(\sum_{i=1}^{k}\epsilon N_iy_{i}-\epsilon N_0y_{0}-q)\\ x_{0}\wedge x_{1}\\ (d_{1}-x_{1})\wedge x_{2}\\ .\\ .\\ .\\ (d_{k-1}-x_{k-1})\wedge x_{k}\\
\end{bmatrix},$$ $$H_3 (\x)=\begin{bmatrix}
	(C_0+\epsilon N_0) x_{0} -\sum_{i=1}^{k} (C_i+\epsilon N_i)x_{i}+(\sum_{i=1}^{k}\epsilon N_iz_{i}-\epsilon N_0z_{0}-q)\\ x_{0}\wedge x_{1}\\ (d_{1}-x_{1})\wedge x_{2}\\ .\\ .\\ .\\ (d_{k-1}-x_{k-1})\wedge x_{k}\\
\end{bmatrix}.$$ By putting $\x=\y$ in $H_2(\x)$, and $\x={\bf z}$ in $H_1(\x)$ and $H_3(\x)$, we get $$H_1({\bf z})=H_2(\y)=H_3({\bf z})=0.$$ For $\epsilon$ is near to zero, deg$(H_1,\Omega,{\bf 0})$= deg$(H_2,\Omega,{\bf 0})$= deg$(H_3,\Omega,{\bf 0})$ due to the nearness property of degree (D3). As ${\bf z} \in \Omega$ is a solution to $H_3 (\x)={\bf 0}$ and ${\bf C+\epsilon N}$ has the column $W$-property, we get deg$(H_3, \Omega, {\bf 0}) \neq 0$  by Theorem \ref{T4} and \ref{degg}. Since deg$(H_2,\Omega,{\bf 0})$= deg$(H_3,\Omega,{\bf 0})$, we have deg$(H_2,\Omega,{\bf 0})\neq 0$. This implies that if we set $ {q_2}:=q+\epsilon N_0 y_{0}-\sum_{i=1}^{k}\epsilon N_i y_{i}$, then EHLCP$({\bf C+\epsilon N},\d,q_2)$ must have a solution in $\Omega$.  As ${\bf C+\epsilon N}$ has the column $W$-property, by Theorem \ref{P1},  EHLCP$({\bf C+\epsilon N},\d,q_2)$ has a unique solution which must be equal to $\y$. So, $\y\in \Omega$. It gives us a contradiction. Hence SOL$({\bf C},\d,q)=A$. Thus SOL$({\bf C},\d,q)$ is connected.  
\end{proof}
\section{Conclusion}
In this paper, we introduced the $R_0$-$W$ property and SSM-$W$ properties and then studied the existence and uniqueness result for EHLCP when the underlying set of matrices has these properties. Last, we gave a necessary and sufficient condition for the connectedness of the solution set of the EHLCP.
\section*{Declaration of Competing Interest} The authors have no competing interests.
\section*{Acknowledgements} The first author is a CSIR-SRF fellow, and he wants to thank the Council of Scientific \& Industrial Research(CSIR) for the financial support.

\end{document}